\documentclass[a4paper,fleqn,leqno]{article}
\usepackage[utf8]{inputenc}
\usepackage[english]{babel}
\usepackage{latexsym,amssymb,amsthm,xcolor}
\usepackage[centertags]{amsmath}
\usepackage{mathtools}
\mathtoolsset{showonlyrefs}
\usepackage{bbm}
\renewcommand{\mathbb}{\mathbbm}
\usepackage{enumitem}

\usepackage{fancyhdr}
\usepackage{appendix}                    %% http://www.tex.ac.uk/cgi-bin/texfaq2html?label=appendix

\newcommand{\highlight}[1]{\colorbox{yellow}{$\displaystyle #1$}}

\usepackage{hyperref}
\usepackage{accents}           % see http://tex.stackexchange.com/a/125414/16865

\usepackage[colorinlistoftodos,
            textwidth=2.7cm,
            textsize=scriptsize,
            color=yellow!30,
            linecolor=orange,
            bordercolor=orange,
            shadow]{todonotes}

\reversemarginpar                         %% todonotes on the left margin
\newtheorem{theorem}{Theorem}[section]
\newtheorem{proposition}[theorem]{Proposition}
\newtheorem{corollary}[theorem]{Corollary}
\newtheorem{lemma}[theorem]{Lemma}

\makeatletter                                           % This is to ensure that remark titles are bold face (as in theorem
\def\th@newremark{\th@remark\thm@headfont{\bfseries}}   % style) but the body is in normal font instead of italics (as
\makeatletter                                           % in theorem style). The standard remark style of amsthm produces

\theoremstyle{definition}
\newtheorem{definitionx}[theorem]{Definition}

\theoremstyle{newremark}
\newtheorem{remarkx}[theorem]{Remark}

\newenvironment{remark}
  {\pushQED{\qed}\remarkx}
  {\popQED\endremarkx}
\newenvironment{definition}
  {\pushQED{\qed}\definitionx}
  {\popQED\enddefinitionx}

\newcommand{\suppress}[1]{}
\DeclarePairedDelimiter\abs{\lvert}{\rvert}%
\DeclarePairedDelimiter\norm{\lVert}{\rVert}%

\makeatletter
\let\oldabs\abs
\def\abs{\@ifstar{\oldabs}{\oldabs*}}
\let\oldnorm\norm
\def\norm{\@ifstar{\oldnorm}{\oldnorm*}}
\makeatother
\newcommand{\eps}{\varepsilon}

\newcommand{\dx}{\mathrm{d}}                                   % the d for differentials
\DeclareMathOperator{\supp}{supp}                              % support
\newcommand{\eqd}{\overset{\textup{d}}{=}}                     % *e*qual *i*n *d*istribution
\newcommand{\ubar}[1]{\underaccent{\bar}{#1}}                  % makes a bar under the symbol, requires package accents

\newcommand{\R}{\mathbb{R}}
\newcommand{\C}{\mathbb{C}}
 \newcommand{\N}{\mathbb{N}}

 \newcommand{\prob}{\mathbb{P}}
 \renewcommand{\P}{\prob}
 
 \newcommand{\E}{\mathbb{E}}
 \newcommand{\Ex}[2][]{\E_{#1}\bigl[ #2 \bigr]}
 \newcommand{\bEx}[2][]{\E_{#1}\Bigl[ #2 \Bigr]}
\allowdisplaybreaks[4]

\newcommand{\mcB}{\mathcal{B}}
\newcommand{\mcC}{\mathcal{C}}

\newcommand{\mcF}{\mathcal{F}}
\newcommand{\mcG}{\mathcal{G}}

\newcommand{\mcM}{\mathcal{M}}

\newcommand{\mcP}{\mathcal{P}}

\newcommand{\bbN}{\mathbb{N}}

\newcommand{\bbR}{\mathbb{R}}

\newcommand{\linspan}{\operatorname{span}}                     % note that \span is already defined
\newcommand{\1}{\mathbbm{1}}

\newcommand{\cadlag}{c\`{a}dl\`{a}g}

\newcommand{\nn}{\nonumber}
\newcommand{\la}{\langle}
\newcommand{\ra}{\rangle}

\newcommand{\td}{\ubar{t}}
\newcommand{\muh}{\hat{\mu}}
\newcommand{\Bh}{\widehat{B}}
\newcommand{\eh}{\hat{e}}
\newcommand{\psih}{\hat\psi}

\newcommand{\tow}{\xrightarrow{\mathrm{w}}}
\newcommand{\towh}{\xrightarrow{\mathrm{w}^\#}}
\newcommand{\weakh}{weak$^\#$}
\newcommand{\Weakh}{Weak$^\#$}
\newcommand{\whconv}{\weakh-convergence}
\newcommand{\Whconv}{\Weakh-convergence}
\newcommand{\Mh}{\mcM^\#}
\newcommand{\MF}[1][\mcF]{\Mh_{#1}}
\newcommand{\MFX}{\mcM_\mcF(X)}
\newenvironment{proofsteps}{\setcounter{enumi}{0}}{}
\newcommand{\step}{\refstepcounter{enumi}\removelastskip\smallskip\par\noindent\emph{Step \arabic{enumi}.} \hspace{0.5ex}}
\newcommand{\su}{\underline{s}}
\newcommand{\muex}{\mu_{\textrm{exc}}}
\newcommand{\get}{g\bigl(e_1(t_1),\ldots,e_n(t_n)\bigr)}
\newcommand{\tmax}{\bar{t}}

\begin{document}

\title{Boundedly finite measures:\\
	Separation and convergence by an algebra of functions\footnote{Preprint of Electron.\ Commun.\ Probab.\
	\textbf{21} (2016), no.\ 60, 1--16.}}
%{Boundedly finite measures: separation and convergence by an algebra of functions}
\author{Wolfgang L\"ohr$^{1}$, Thomas Rippl$^{2,3}$}
\date{{\today}}
\maketitle

\begin{abstract}
  We prove general results about separation and \whconv\ of boundedly finite measures on separable metric spaces
  and Souslin spaces. More precisely, we consider an algebra of bounded real-valued, or more generally a
  $*$\nobreakdash-algebra $\mcF$ of bounded complex-valued functions and give conditions for it to be separating or
  \whconv\ determining for those boundedly finite measures that integrate all functions in $\mcF$. For
  separation, it is sufficient if $\mcF$ separates points, vanishes nowhere, and either consists of only
  countably many measurable functions, or of arbitrarily many continuous functions. For convergence determining,
  it is sufficient if $\mcF$ induces the topology of the underlying space, and every bounded set $A$ admits a
  function in $\mcF$ with values bounded away from zero on $A$.
\end{abstract}
\bigskip

\smallskip


\footnotetext[1]{University of Duisburg-Essen, Mathematics, e-Mail: wolfgang.loehr@uni-due.de}
\footnotetext[2]{Institut f\"ur Math.~Stochastik, Universit\"at G\"ottingen, e-Mail: trippl@uni-goettingen.de}
\footnotetext[3]{Supported by DFG SPP 1590}

\renewcommand{\abstractname}{Contents}
\begin{abstract}
	\vspace*{-3.75\baselineskip}
	\renewcommand{\contentsname}{}
	\tableofcontents
\end{abstract}

%\smallskip
%\begin{center}\begin{minipage}{0.883\textwidth}
%{
%  \footnotesize
%  \def\section{\subsection}
%  \listoftheorems[ignoreall,show={theorem,corollary,proposition,lemma}]
%  \smallskip
%  \tableofcontents
%  \smallskip
%}
%\end{minipage} \end{center}

%%%%%%%%%%%%%%%%%%%%%%%%%%%%%%%%%%%%%%%%%%%%%%%%%%%%%%%%%%%%%%%%%%%%%%%%%%%%%%%%%%%%%%%%%%%%%%%%%%%%
%**************************************** MAIN PART OF DOCUMENT ****************************************
%%%%%%%%%%%%%%%%%%%%%%%%%%%%%%%%%%%%%%%%%%%%%%%%%%%%%%%%%%%%%%%%%%%%%%%%%%%%%%%%%%%%%%%%%%%%%%%%%%%%

\section{Introduction}
\setcounter{equation}{0}
% \numberwithin{equation}{section}

Boundedly finite measures play an increasingly important role in probability theory.
Classical examples are It\^o excursion measures, or L\'evy measures, which we will come back to in
a moment. But they also appear in very recent research as speed measures of Brownian motions on $\R$-trees
\cite{AEW13,ALW15} or sampling measures for spatial Fleming-Viot processes \cite{GSW15}. Convergence
of metric measure spaces with boundedly finite measures has been analysed with a view towards probabilistic
applications in \cite{ALW14a}.

For the purpose of illustration, we quickly recall the situation for L\'evy measures of infinitely divisible
random variables.
A real-valued random variable $X$ is called \emph{infinitely divisible} if, for any $n \in \bbN$, we can write
$X \eqd X_{1,n}+\dots + X_{n,n}$ for an i.i.d.~sequence of random variables $(X_{i,n})_{1\leq i \leq n}$.
A famous theorem by L\'evy and Khintchine states that in that case the Fourier transform of the random variable has a very explicit form:
\begin{equation}\label{e.tr1}
 -\log \E \bigl[ \exp(-iuX)\bigr]  = iub + \frac{c}{2} u^2 - \int_{\bbR \setminus \{0\}}  e^{-iuy} - 1 +iuy
 \1_{|y|\leq 1} \;  \mu(\dx y) ,\quad u \in \R,
\end{equation}
where $b\in \R$, $c \geq 0$ and $\mu$ is a measure on $\bbR \setminus \{0\}$ satisfying $\int_{\bbR
\setminus \{0\}} 1\wedge |y|^2\,\mu(\dx y) < \infty$, see Theorem 8.1 of \cite{Sato}.
Let us look at the measure $\mu$, called \emph{L\'evy measure}.
The obvious questions are: (a) is the L\'evy measure  unique?\footnote{this question has an
affirmative answer with a neat proof in \cite{berg_positive_1976}  Theorem 3.7}
(b) If we have a sequence of infinitely divisible random variables, do their L\'evy measures converge and in
which sense?
L\'evy measures are not necessarily finite measures, but are required to be finite on any set which is not close to the origin $0$. This motivates to consider measures which are finite on a certain class of sets.
Daley and Vere-Jones used Appendix A2.6 in \cite{DVJ03} to present a framework for such questions which they
call \emph{boundedly finite measures}, because the measures are assumed to be finite on bounded sets. The space
of these measures is equipped with \whconv, defined as convergence of integrals over bounded continuous
functions with bounded support (see Section~\ref{s.main} for definitions).
Some extensions are given in \cite{hult2006regular} and \cite{lindskog2014}.
L\'evy measures fit into this framework if we change the Euclidean metric on $\R\setminus\{0\}$ such that $0$ is
sent infinitely far away, an idea also used in \cite{Barczy20061831}.

\smallskip

How does one prove weak convergence $\mu_n \tow \mu$ of probability measures on a topological space $X$ in
situations where it is not feasible to show convergence of $\int f \,\dx \mu_n$ for \emph{all} bounded
continuous $f\in \mcC_b(X)$ directly? One possibility is to find a class $\mcF\subseteq\mcC_b(X)$ of
sufficiently ``nice'' functions, which is still rich enough to be \emph{convergence determining}, i.e.
\begin{equation}
\int f \, \dx \mu_n \to \int f \, \dx \mu \ \forall \, f \in \mcF
 	\quad\implies\quad \mu_n \tow \mu .
\end{equation}
This approach has proven to be particularly fruitful if the topology on $X$ itself is defined in terms of a
class of functions. A classical example for such a topology is the weak (weak\nobreakdash-$*$) topology on (the dual of) a
Banach space. A more modern one is the Gromov-weak topology on the space of metric measure spaces, which is
induced both by a complete metric and by a class of functions called \emph{polynomials} (see \cite{GPW09}).
That polynomials do not only induce the topology but are even convergence determining was shown with some effort
in \cite{DGP11}. But it also follows directly from a general result due to Le Cam, as pointed out in
\cite{L13}. Le Cam's theorem goes back to \cite{LeCam} and states that on a completely regular Hausdorff
space $X$, a set of functions $\mcF\subseteq\mcC_b(X)$ is convergence determining for Radon probability measures
if it is multiplicatively closed and induces the topology of $X$. The proof can be found in
\cite[Proposition~4.1]{HoffJ76}.
A version of Le Cam's theorem for separable metric spaces dropping the ``Radon'' assumption on the probability
measures is given in \cite{BlountKouritzin10}. This version was used extensively for the construction of a
tree-valued pruning process in \cite{LVW15}.

Our main goal is to extend Le Cam's result to the case of boundedly finite measures and \whconv\ and, because
convergence determining is sometimes too much to ask for, to obtain (weaker) sufficient conditions for $\mcF$ to
at least \emph{separate} boundedly finite measures. A separating class of functions can also be used to prove
weak (or \weakh) convergence if tightness is known by other methods.
In particular, our results allow to give an answer on the question of uniqueness and convergence of $\mu$ in
\eqref{e.tr1} within a general framework. More importantly, they will find applications in future work about
spaces of metric measure spaces and $\R$-trees such as in the upcoming paper \cite{infdiv} which was a driving
motivation for the present article. The results will hopefully also facilitate the analysis of spatial
population models on unbounded spaces with infinite total population size such as the one in \cite{GSW15}, and
of other models appearing in modern areas of probability theory.

\suppress{
\smallskip

To compare our situation to that of \cite{lindskog2014}, assume that $C \subset X$ is closed and set $d(x,C) = \inf_{y \in C} d(x,y)$.
If we use the metric $(x,y)\mapsto d(x,y) + \abs{ d(x,C)^{-1} - d(y,C)^{-1} } =: d_C(x,y)$ on $X \setminus C$, then their setup fits into ours.\footnote{the reader may verify that this is a metric}
Obviously our setup also fits in their situation with letting $C= \emptyset$.

There has been some work on measures which are restricted to be finite on \emph{compact} sets, see
e.g.~\cite{resnick2007heavy}, page 50ff, but we do not consider that wider class of measures.%, except Remark \ref{r.compact}.
}

\smallskip

The rest of the paper is organized as follows. In Section~\ref{s.main}, we give our main results about
convergence determining (Theorem~\ref{t.main}) and separating (Theorem~\ref{t.sep} and Corollary~\ref{c.main})
classes of functions for boundedly finite measures.
In Section~\ref{s.examples}, we illustrate in four examples how our results can be applied. There, we consider
L\'evy measures, excursion theory and mass fragmentations.
% 
% The rest of the article is organized in a very easy way.
% First, we give precise definitions and state the two main results in Section \ref{s.def+result}.
% The final section contains the proofs of the results.
% Proposition \ref{p.Portmanteau} is proven in Subsection \ref{ss.portmanteau}.
% Thereafter we give some auxiliary results for the proof of Theorem \ref{t.main} in Subsection \ref{ss.aux} and finally that theorem is shown in Subsection \ref{ss.SWproof}.
% 

%**************************************** Section: Main Result ****************************************
\section{Separation and convergence of boundedly finite measures}\label{s.main}

Let $(X,d)$ be a separable metric space, endowed with the Borel $\sigma$-field induced by $d$.
By $\mcC_b(X)$, we denote the set of bounded continuous functions on the metric space $(X,d)$ with values in
$\C$. For real-valued functions we write $\mcC_b(X;\R)$. Note that $\mcC_b(X;\R)\subseteq \mcC_b(X)$.

\begin{definition}[Boundedly finite measures and \whconv]
The set of \emph{boundedly finite measures} $\mcM^\#(X)$ on $X$ w.r.t.~$d$ is given as
\begin{equation}\label{e.Mh}
  \mcM^\#(X) = \left\{ \mu \in \mcM(X) \mid \mu(A) < \infty \text{ for all $d$-bounded, measurable } A\subseteq X\right\}.
\end{equation}
A sequence $(\mu_n)_{n\in\N}$ in $\Mh(X)$ is
said to be \emph{\weakh-convergent} to $\mu\in\Mh(X)$, denoted by $\mu_n\towh \mu$, if $\int f\,\dx\mu_n\to \int
f,\dx\mu$ holds for all $f\in\mcC_b(X;\R)$ with $d$-bounded support.
\end{definition}

\begin{remark}[\Whconv\ versus vague convergence]
If $(X,d)$ is a \emph{Heine-Borel space}, i.e.\ every closed, bounded set is compact, then $\Mh(X)$ coincides with
the set of Radon measures on $X$, and \whconv\ with vague convergence. For a general separable metric space,
however, $\Mh(X)$ is a subset of the Radon measures and \whconv\ is a potentially much stronger convergence than
vague convergence.
\end{remark}

\noindent
Consider a set $\mcF$ of measurable, $\C$-valued functions on $X$ and define
\begin{equation}\label{mcF}
 \mcM_\mcF (X) = \Bigl\{ \mu \in \mcM(X) \Bigm| \int \abs{f(x)}\, \mu(\dx x) < \infty \ \;\forall \, f \in \mcF\Bigr\},
 	\quad \MF(X) = \mcM^\#(X) \cap \mcM_\mcF(X).
\end{equation}

\begin{theorem}[Convergence determining for boundedly finite measures]\label{t.main}
Let\/ $(X,d)$ be a separable metric space and\/ $\mcF \subset \mcC_b(X)$.  Assume that 
  \begin{enumerate}[label=\normalfont(\textbf{T.\arabic*}), ref = {T.\arabic*}]
 \item\label{i.mult} $\mcF$ is multiplicatively closed and %(in the $\C$-valued case)
 closed under complex conjugation.
 \item\label{i.top} $\mcF$ induces the topology of\/ $X$.
 \item\label{i.bound} For every bounded set\/ $A \subset X$ there exists\/ $f \in \mcF$ and\/ $\delta >0$ with\/
	 $\inf_{x \in A} |f(x)| > \delta$.
\end{enumerate}
Then\/ $\mcF$ is weak$^\#$-convergence determining for measures in\/ $\mcM_\mcF^\#(X)$, i.e.
\begin{equation}
 \mu,\mu_n \in \mcM_\mcF^\#(X),\, \int f \, \dx \mu_n \to \int f \, \dx \mu \ \forall \, f \in \mcF
 	\quad\implies\quad \mu_n \towh \mu .
\end{equation}
\end{theorem}

\begin{remark}
 \begin{enumerate}
  \item \eqref{i.mult} and \eqref{i.top} are classical assumptions for these kind of theorems, see
	Proposition~4.1 in \cite{HoffJ76}. Something like \eqref{i.bound} is necessary to replace the fixed
	total mass in the weak convergence of probabilities. At least, we have to ensure that $\mcF$ ``vanishes
	nowhere'', because if there was $x\in X$ with $f(x)=0$ for all $f\in\mcF$, then $\mcF$ could not even
	separate $a\cdot \delta_x$ for different $a\ge 0$. For the purpose of separation of measures, we can do
	with this weaker requirement \ref{i.vanish} in Theorem~\ref{t.sep} below. We do not know, however, if
	it would be enough for Theorem~\ref{t.main}.
  \item For real-valued functions, the part ``closed under complex conjugation'' is always satisfied.
%  \item We need to assume that $\mu \in \mcM^\#_\mcF(X)$. A counterexample where the result does not apply with $\mu \not \in \Mh_\mcF$ is given in \colorbox{yellow}{XXXX}.
%  \item \eqref{i.bound} is necessary since: counterexample or explanation (reconstruction of constant function).
  \qedhere
 \end{enumerate}
\end{remark}

\noindent
For the proof, we embed everything in the Hilbert cube $H$, a technique going back to Urysohn's work on
metrisation and also used in \cite{BlountKouritzin10}. Recall that
\[ H= [0,1]^\bbN, \quad\text{with product topology.} \]
Denote the uniform norm on $\mcC_b(H;\R)$ by $\norm{\cdot} := \norm{\cdot}_\infty$.
For $0< \delta <1$, we consider the subspace 
\[ H_\delta := \bigl\{x=(x_n)_{n\in\N} \in H \bigm| x_1 \geq \delta \bigr\} \]
and use the following variant of the Stone-Weierstrass theorem.

\begin{definition}[$\mcP$, $\mcP_0$]\label{d.poly}
 Let $\mcP \subseteq \mcC_b(H;\R)$ be the set of polynomials on $H$ (i.e.~functions depending on finitely many
 coordinates and an algebraic multivariate polynomial in these coordinates).
 Let $\mcP_0 :=  \bigl\{ p\in \mcP \bigm| p(x) = 0 \; \forall \, x=(x_n)_{n\in\N} \in H \text{ with\/
 $x_1=0$}\bigr\}$.
\end{definition}

\begin{lemma}[Stone-Weierstrass variant]\label{l.SW}
 Let\/ $g \colon H \to [0,1]$ be continuous with\/ $\supp g \subset H_\delta$ for some\/ $\delta >0$.
 Then, for every\/ $\eps >0$ there exists a polynomial\/ $p_\eps \in \mcP_0$ such that
 \begin{equation}
   \abs{g(x) - p_\eps(x) } \leq \eps x_1 \quad \forall \, x=(x_n)_{n\in\N} \in H .
 \end{equation}
\end{lemma}
\begin{proof}
 For $x=(x_n)_{n\in\N} \in H$ define
 \begin{equation}
  \tilde{g}(x) := \begin{cases}
                       x_1^{-1} g(x) \ & \text{ if } x_1 >0, \\
                       0 \ & \text{ if } x_1 = 0.
                      \end{cases}
 \end{equation}
 Then $\tilde{g} \in \mcC_b(H;\R)$ and we can use the Stone-Weierstrass theorem.
 So, for $\eps >0$ we find $\tilde{p}_\eps \in \mcP$ such that
 \[ | \tilde{g}(x) - \tilde{p}_\eps(x) | < \eps \ \forall x \in H .\]
 Define $p_\eps (x) := x_1 \tilde{p}_\eps(x)$, $x \in H$.
 Then $p_\eps \in \mcP_0$ and we get for any $a\in (0,1]$ the estimate
 \begin{equation}
  \sup_{x\in H, \; x_1=a} a^{-1} \abs{g(x) - p_\eps(x)} = \sup_{x\in H, \; x_1=a} \abs{ \tilde{g} (y) -
  \tilde{p}_\eps(y) } < \eps.
 \end{equation}
 That is what we needed to show, as the case $x_1=0$ is trivial.
\end{proof}

%Now we can prove the main result.

\begin{proof}[Proof of Theorem \ref{t.main}]
\begin{proofsteps}
 \step It is enough to consider functions with values in $[0,1]$:
    First, $\mcF$ may be replaced by $\mcF':=\bigl\{\Re(f),\, \Im(f)\bigm| f\in\mcF\bigr\}$, where $\Re(f)$ and
    $\Im(f)$ are the real and imaginary parts of $f$, respectively. Because $\mcF$ is closed under complex
    conjugation due to \eqref{i.mult}, the conditions \eqref{i.mult}, \eqref{i.top} and \eqref{i.bound} are
    also satisfied for $\mcF'$ instead of $\mcF$. Thus we may assume $\mcF$ to consist of real-valued functions.
    Second, $\mcF$ may be replaced by $\mcF':=\{f^2\mid f\in \mcF\} \cup
    \bigl\{f^2(\|f\|_\infty-f)\bigm| f\in \mcF\bigr\}$, which is contained in the vector space generated by
    $\mcF$ and easily seen to satisfy the prerequisites of the theorem provided that $\mcF$ does.
    Because $\mcF'$ maps to $\R_+$, we can assume, by normalisation, that the elements of $\mcF$ map into
    $[0,1]$.
    % It is enough to consider the $\R$-valued case, because in the $\C$-valued case, we consider the set
    % $\mcF':=\bigl\{\Re(f),\, \Im(f)\bigm| f\in\mcF\bigr\}$ and observe that the prerequisites are satisfied with
    % $\mcF'$ instead of $\mcF$.

 \step\label{s.subspace}
 %Assume w.l.o.g.~that $\mcF$ is a vector space, and hence an algebra by Assumption~\eqref{i.mult}.
    By Assumption \eqref{i.top}, $\mcF$ induces the topology of $X$. Because $X$ is a separable, metric space,
    it has a countable basis, and thus we can choose a countable subfamily of $\mcF$ that still induces the
    topology of $X$. Indeed, the family of sets of the form $\bigcap_{i=1}^n f^{-1}_i(U_i)$ for $n\in\N$,
    $f_i\in \mcF$, $U_i\subseteq[0,1]$ open is a base for the topology, and because $X$ has a countable base, we
    can select a countable subfamily that is still a base\footnote{One can show this with standard arguments: If
    $\mcB,\mcB'$ are bases, $\mcB$ countable, $B\in \mcB$, then $B=\bigcup I$ for some $I\subseteq \mcB'$, and
    for every $U\in I$ there is $J_U\subseteq \mcB$ with $U=\bigcup J_U$. For every $V\in J:=\bigcup_{U\in I}
    J_U$, we select one $U_V\in I$ with $V\in J_{U_V}$. $J$ is a subset of $\mcB$, hence countable.
    Because $B=\bigcup J=\bigcup_{V\in J} U_V$, we obtain a countable basis by taking all $U_V$ for all joices
    of $B\in \mcB$.}.
    Therefore there exist $f_1, f_2, \ldots \in \mcF$ with $0 \leq f_m \leq 1$, such that
    $(f_m)_{m\in \N}$ induces the topology of $X$.
    Then $\iota\colon X \to H$, $x \mapsto (f_m(x))_{m \in \N}$ is a topological embedding (i.e.\ a
    homeomorphism onto its image) of $X$ into $H$.
    Identifying $X$ with $\iota(X)$, we assume w.l.o.g.~$X \subseteq H$ and $f_n$ to be the (restriction of the)
    $n^{\textrm{th}}$ coordinate projection.  In particular, being an algebra, the linear span of $\mcF$
    contains $\mcP_0$ (defined in Definition~\ref{d.poly}).

 \step Let $\mu_n, \mu \in \mcM_\mcF^\#(X)$ with
    \begin{equation}\label{conv}
	 \int f\, \dx \mu_n \to \int f \, \dx \mu \quad \forall \, f \in \mcF .
    \end{equation}
    For the claimed \whconv, it is enough to show $\int g \, \dx \mu_n \to \int g \, \dx \mu$ for all
    $g\colon X \to [0,1]$ which have $d$-bounded support and are uniformly continuous (by the Portmanteau
    theorem, Theorem 2.1 in \cite{lindskog2014}).
    Because weak convergence depends on the metric only through the induced topology, we may assume $g$ to be
    uniformly continuous w.r.t.\ any other metric on $X$ inducing the same topology as $d$. To this end, we
    take any metric on $H$ inducing its topology, and assume that $g$ is uniformly continuous w.r.t.\ its
    restriction to $X$ (recall that $X$ is a subspace of $H$ by Step~\ref{s.subspace}).
    By Assumption~\eqref{i.bound}, there is $f \in \mcF$ and $\delta>0$ such that for all $x\in\supp(g)$ we have
    $|f(x)| > \delta$. We may assume w.l.o.g.\ that $f=f_1$ (if not, we define $f_1'=f$, $f_{m+1}'=f_m$ and
    observe that $\iota'$ defined with these $f'_m$ is still an embedding and $g$ uniformly continuous w.r.t.\
    the restriction of the metric on $H$). 
    Then $\supp(g) \subseteq H_{\delta}$. Furthermore, since $g$ is uniformly continuous, it can be extended
    continuously to the closure of $X$ in $H$, and by the Tietze extension theorem (e.g.\
    \cite[Theorem~35.1]{Munkres00}) % also Hitchhiker 2.47
    to a continuous function from $H$ to $[0,1]$ with support in $H_{\delta}$.
    We denote the extension again by $g$. We also identify $\mu_n$ and $\mu$ with their natural extensions to $H$.

 \step $g$ satisfies the assumptions of Lemma~\ref{l.SW}. For $\eps>0$, choose $p_\eps \in \mcP_0 \subseteq
	\linspan(\mcF)$ as in the lemma. Because $\mu_n (f_1) \to \mu(f_1)$, we have $M := \sup_{n \in \N}
	\mu_n(f_1) < \infty$ and obtain for all $n\in\N$
	\[ \abs{ \int p_\eps - g  \, \dx \mu_n } \leq \int_{H} \eps x_1 \, \mu_n(\dx x)
		= \eps \int f_1\, \dx \mu_n \leq \eps M. \]
    Because $\mu_n(p_\eps) \to \mu(p_\eps)$ by \eqref{conv}, we conclude for every $\eps>0$
 \begin{equation}
  \limsup_{n \to \infty} \abs{ \mu_n(g) - \mu(g) } \leq
  	 \limsup_{n\to\infty} \abs{\mu_n(p_\eps) - \mu(p_\eps)} + \abs{ \mu_n(p_\eps - g) } + \abs{ \mu (p_\eps -g) }
	\leq  2\eps M.
 \end{equation}
    Because $\eps$ is arbitrary, the claimed convergence follows.
\end{proofsteps}
\end{proof}

It is desirable to have a result which separates two boundedly finite measures but requires less than the previous theorem.
While ``boundedly finite'' is essential for the definition of \whconv, we can drop this assumption for the
purpose of separation and work with $\mcM_\mcF$, the space of measures integrating $\mcF$ as defined in
\eqref{mcF}, instead of $\MF$. We can also relax
the metrisability assumption on the space $X$, but do need some topological assumption.
Recall that a topological space is Hausdorff if any two distinct points can be separated by open sets,
and a Hausdorff topological space $X$ is, by definition, a \emph{Souslin space} if there exists a Polish space
$Y$ and a continuous surjective map from $Y$ onto $X$. Note that a Souslin space is separable but need not be
metrisable. An example is a separable Banach space in its weak topology, which is clearly Souslin but not
metrisable. Conversely, not every separable metrisable space is Souslin. In the case of a
Souslin space $X$, and a countable family of functions $\mcF$, we can drop the topological assumptions on $\mcF$
from the prerequisites of Theorem~\ref{t.main} and still obtain the weaker conclusion that $\mcF$ is separating
for measures in $\mcM_\mcF$. More precisely, we have

\begin{theorem}[Separation of boundedly finite measures with measurable functions]\label{t.sep}
Let\/ $X$ be a Souslin space (for example a Polish space), and\/ $\mcF$ a \emph{countable} set of bounded,
measurable\/ $\C$\nobreakdash-valued functions. Assume that 
\begin{enumerate}[label=\normalfont(\textbf{S.\arabic*}), ref={S.\arabic*}]
% \item\label{i.countcont} $\mcF$ is either countable or $\mcF\subseteq \mcC_b(X)$
 \item\label{i.vmult} $\mcF$ is multiplicatively closed and closed under complex conjugation.
 \item\label{i.sp} $\mcF$ separates points of\/ $X$.
 \item\label{i.vanish} $\mcF$ vanishes nowhere,
	i.e.\ for every\/ $x\in X$ there exists an\/ $f_x \in \mcF$ with\/ $f_x(x)\ne 0$.
\end{enumerate}
Then\/ $\mcF$ is separating for measures in\/ $\mcM_\mcF(X)$, i.e.
 \begin{equation}\label{e.sep}
 \mu_1, \mu_2 \in \mcM_\mcF(X), \, \int f \, \dx \mu_1 = \int f\, \dx \mu_2 \ \forall \, f \in \mcF
 \quad\implies\quad \mu_1 = \mu_2 .
 \end{equation}
\end{theorem}

\begin{proof}
Assume that $\mu_1,\mu_2\in \MFX$ are such that $\int f\,\dx\mu_1=\int f\,\dx\mu_2$ holds for all $f\in \mcF$. We
have to show $\mu_1=\mu_2$.
Enumerate $\mcF=\{f_n\mid n\in\N\}$. Using Step 1 from the proof of Theorem~\ref{t.main}, we may (and do)
assume that $f_n$ takes values in $[0,1]$ for all $n\in\N$.
We proceed in two steps: first we show that \eqref{e.sep} holds if we assume that instead of \eqref{i.vanish}
the following stronger condition holds,
\begin{enumerate}[label=\normalfont(\textbf{S.\arabic*$^\dagger$}), ref =S.\arabic*$^\dagger$]
 \setcounter{enumi}{2}
 \item\label{i.pos} $f_1(x) \ne 0$ for all\/ $x\in X$.
\end{enumerate}
In the second step, we reduce the general case to the one where \eqref{i.pos} holds.
\begin{proofsteps}
\step\label{step.1} Assume that \eqref{i.pos} holds and define $\iota\colon X \to H$,
	$x \mapsto \bigl(f_n(x)\bigr)_{n\in\N}$.
	Then $\iota$ is measurable and injective by assumption \eqref{i.sp}.
%	Let $\muh_i:=\mu\circ\iota^{-1}$, $i=1,2$, be the image measures under $\iota$.
	Because $X$ is a Souslin space, it is an analytic measurable space (\cite[Proposition~8.6.13]{Cohn80})
	and so is $Y:=\iota(X)$ (\cite[Corollary~8.6.9]{Cohn80}).
	By \cite[Proposition~8.6.2]{Cohn80}, $\iota$ is a Borel isomorphism onto $Y$, i.e.\ $\iota^{-1}\colon Y
	\to X$ is measurable. Therefore,
	\begin{equation}\label{e.isom}
		\mu_1=\mu_2 \;\iff\; \mu_1\circ\iota^{-1}=\mu_2\circ\iota^{-1}.
	\end{equation}
%	Therefore, by identifying $X$ and $Y$ with $\iota$, we may assume w.l.o.g.\ $X\subseteq \R_+^\N$
%	Observe $\pi_n\in\mcF\subseteq L^1(\mu+\nu)$.
%	Let $A \subseteq \pi_1^{-1}\bigl([\delta, \infty)\bigr)$
%Define $\Xh:=\bigl\{x=(x_n)_{n\in\N}\in \R^\N \bigm| x_1 \ne 0\bigr\}$.
	%Because of \eqref{i.pos}, we can consider $\muh_i$ as measures on $\Xh$.
	Because of \eqref{i.pos}, every $x=(x_n)_{n\in\N}\in Y$ satisfies $x_1\ne 0$. We define the metric
		\[ r(x,y):=|x_1^{-1}-y_1^{-1}| + \sum_{n\in\N} 2^{-n} |x_n - y_n| \land 1, \qquad x,y\in Y,\]
	which induces on $Y$ the topology inhereted as a subspace of $H$.
	Let $\mcG:= \{f\circ\iota^{-1} \mid f\in \mcF\}$. We show that $(Y,r)$ and $\mcG$ satisfy the
	prerequisites of Theorem~\ref{t.main}.
	
	$\mcG$ satisfies \eqref{i.mult} because $\mcF$ does by assumption.
	By construction of $\iota$, $\mcG$ coincides with the set of restrictions
	of coordinate projections to $Y$. Therefore, $\mcG$ induces the topology of $Y$, i.e.\ \eqref{i.top} is
	satisfied. An $r$-bounded set $A$ in $Y$ satisfies $\delta:=\inf_{x\in A} |x_1|>0$, and \eqref{i.bound}
	is satisfied with $f=f_1\circ\iota^{-1}$. Thus we can apply Theorem~\ref{t.main} and obtain that $\mcG$
	is \whconv\ determining and a fortiori separating for measures in $\MF[\mcG](Y)$.

	If $\mu_1,\mu_2\in \MFX$, then they integrate $f_1$, and $\muh_i:=\mu_i\circ\iota^{-1}$, $i=1,2$, are
	boundedly finite measures on $(Y,r)$. Thus obviously $\muh_i\in \MF[\mcG](Y)$ and the claim of the
	theorem follows with \eqref{e.isom}.

\step\label{step.2} Now consider the general case, where \eqref{i.pos} does not necessarily hold. Define 
	\[ f_0 := \sum_{n\in\N} 2^{-n}  \frac{f_n}{1 \vee \int f_n\,\dx\mu_1}, \]
	and let $\mcF'$ be the set of finite products of elements of $\mcF \cup \{f_0\}$.
	Then $\mcF'$ is a countable set of measurable functions satisfying \eqref{i.vmult}, \eqref{i.sp} and,
	because $\mcF$ vanishes nowhere, also \eqref{i.pos} (with $f_1$ replaced by $f_0\in\mcF'$). Hence, by
	Step~\ref{step.1}, $\mcF'$ is separating for measures in $\mcM_{\mcF'}(X)$.

	According to the monotone convergence theorem, we have
	\[ \int f_0\,\dx\mu_1 =\int f_0\,\dx\mu_2 \leq 1. \]
	Because every element of $\mcF'$ is dominated by an element of $\mcF \cup \{f_0\}$ (recall that
	elements of $\mcF'$ map to $[0,1]$), this implies $\mu_1,\mu_2\in \mcM_{\mcF'}(X)$. 
	%Moreover, for $g \in \mcF'$, dominated convergence yields
	Moreover, dominated convergence yields
	\[ \int g\,\dx\mu_1 =\int g\,\dx\mu_2  \quad \forall g \in \mcF' ,\]
	which implies $\mu_1 = \mu_2$ by Step \ref{step.1}.
	%Again using dominated convergence, we see that $\int h\,\dx\mu_1=\int
	%h\,\dx\mu_2$ for all $h\in\mcG$, and hence $\mu_1=\mu_2$.
\end{proofsteps}
\end{proof}

In the case of \emph{continuous} functions, we can drop the countability of $\mcF$.

\begin{corollary}[Separation of boundedly finite measures with continuous functions]\label{c.main}
Let\/ $X$ be a Souslin space (e.g.\ a Polish space), and\/ $\mcF\subseteq \mcC_b(X)$.
Assume \eqref{i.vmult}, \eqref{i.sp}, and \eqref{i.vanish} from Theorem~\ref{t.sep}.
Then\/ $\mcF$ is separating for measures in\/ $\mcM_\mcF(X)$, i.e.\ \eqref{e.sep} holds.
\end{corollary}
\begin{proof}
For $x\in X$, let $f_x$ be as in \eqref{i.vanish}. There is an open neighbourhood $U_x$ of $x$ with $f_x(y) \ne
0$ for all $y\in U_x$. Recall that a topological space is called Lindel\"of if every open cover has a countable
subcover, and every Polish space has this property (because it has a countable base). Because the property is
obviously preserved by continuous maps, every Souslin space is Lindel\"of as well.
Hence $(U_x)_{x\in X}$ has a countable subcover, and there exists a countable subfamily $\mcF_1$ of $\mcF$
satisfying \eqref{i.vanish}.

Similarly, for $x,y\in X$ let $f_{xy}\in \mcF$ be such that $f_{xy}(x) \ne f_{xy}(y)$. Then there is an open
neighbourhood $U_{xy}$ of $(x,y)$ in $X^2$ with $f_{xy}(u) \ne f_{xy}(v)$ for all $(u,v) \in U_{xy}$.
Because $X^2$ is also Souslin and hence Lindel\"of, we find a countable subfamily $\mcF_2$ of $\mcF$ satisfying
\eqref{i.sp}. Let $\mcF'$ be the closure of $\mcF_1 \cup \mcF_2$ under multiplication and complex conjugation.
Then $\mcF'$ satisfies the prerequisites of Theorem~\ref{t.sep} and the conclusion follows.
\end{proof}

%**************************************** Section: Examples ****************************************
\section{Examples}\label{s.examples}

\subsection{Example 1: L\'evy-Khintchine formula on \texorpdfstring{$\R^D$}{RD}}

Let $Z$ be an infinitely divisible random variable with values in $\R^D$ for $D \in \N$.
That means for any $n \in \bbN$ there are i.i.d.\ random variables $Z_{1,n}, \dotsc, Z_{n,n}$ such that $Z \eqd Z_{1,n}+  \dots + Z_{n,n}$.
Consider 
 \[ X:=\R^D\setminus\{0\} \quad\text{with metric}\quad d(x,y) := \norm{x-y}_\infty + \abs{ \norm{x}_\infty^{-1} - \norm{y}_\infty^{-1}}, \; x,y
 \in X. \]
%\colorbox{yellow}{Why do we take different norms on $\R^D$ for the two terms?}
%Note that $d$ induces the Euclidean topology on $X$, but its bounded sets are precisely those that are bounded away from zero.
It is well-known that there exist $b \in \R^D$, $C \in S_+(\R^D)$ a symmetric, positive semidefinite matrix, and $\mu \in \mcM^\#(X)$  with
$\int (1\wedge \norm{x}^2_\infty) \, \mu(\dx x) < \infty$ such that
%\begin{equation}\label{e.LK}
%	L(u):= -\log \E[ \exp(- u^t Z) ] = u^tb  + \int_{[0,\infty)^d \backslash \{0\}} \left(1-\exp(-u^t x) \right)\, \mu(\dx x) ,
%	\ u \in [0,\infty)^d .
%\end{equation}
\begin{equation}\label{e.FT}
  \Psi(u) := \log \E[ \exp(i u^t Z) ] = i u^tb - \frac{1}{2} u^t C u + \int_{\R^D \setminus \{0\}} \exp(iu^t x) - 1 - iu^tx \1_{|x|\leq 1}\; \mu(\dx x) , \ u \in \R^D .
\end{equation}
This formula is called the L\'evy-Khintchine formula, see \cite[Theorem 8.1]{Sato}.
%The log-Laplace transform $[0,\infty)^d \to \R, u \mapsto L(u)$ characterizes the distribution of the random vector $Z$.
The function  $\R^D \to \C,\, u \mapsto \Psi(u)$ characterizes the distribution of the random
vector $Z$. On a first glance, however, it is not clear why the L\'evy triple
$(b,C,\mu)$  should be unique.
Theorem \ref{t.main} allows to give simple verification of that known fact in a general setup.

\begin{proposition}[Uniqueness and convergence of L\'evy measures]\label{p.LK}
 \eqref{e.FT} determines the L\'evy triple\/ $(b,C,\mu) \in \R^D \times S_+(\R^D) \times \mcM^\#(X)$ uniquely.
% The L\'evy triple\/ $(b,C,\mu) \in \R^D \times S_+(\R^D) \times \mcM^\#(X)$ in \eqref{e.FT} is unique.
 Furthermore, if\/ $Z_n$ are infinitely divisible random variables converging in distribution to\/ $Z$,
 and\/ $(b_n,C_n,\mu_n)$ is the L\'evy triple of\/ $Z_n$, then\/
% as in \eqref{e.FT} with\/ $Z_n$ instead of\/ $Z$, then\/
 %$b_n\to b$, $C_n\to C$ and\/ 
 $\mu_n \towh \mu$.
 %
 %\colorbox{red}{Not true: $b_n$ may jump because of the indicator.}
\end{proposition}
\begin{proof}
\begin{proofsteps}
We start with the unique identification of the law.
 \step First note that
      \begin{equation}
      C_{k,j} = - \lim_{m \to \infty} m^{-2} \left[\Psi(m(\text{e}_k +\text{e}_j)) - \Psi(m \text{e}_k) -\Psi(m\text{e}_j) \right], \ 1\leq k,j \leq D,
      \end{equation}
      where $\text{e}_k$, $k=1,\dotsc, D$ are unit vectors in $\R^D$.
      Moreover, for $k=1,\dotsc, D$:
      \[ b_k = \lim_{m\to \infty} -\frac{i}{m} \left( \Psi(m \text{e}_k) -\frac{1}{2}m^2C_{k,k} \right) . \]
      Hence $C$ and $b$ are unique.
 \step Now suppose
      $\mu_1,\mu_2\in\Mh(X)$ both satisfy $\int (1\wedge \norm{x}^2_\infty) \, \mu_i(\dx x) < \infty$ and \eqref{e.FT}
      with $\mu$ replaced by $\mu_i$, $i=1,2$.

      For $u\in \R^D$, define $F_u,\psi_u,G_u\colon \R^D\to\C$ by
      \begin{equation}\label{e.psi_u} F_u(x) := \exp(iu^t x)-1,\qquad \psi_u(x) := iu^tx\1_{|x|\leq 1}, \qquad G_u(x) := F_u(x) - \psi_u(x), \ x \in \R^D \end{equation}
      and consider the following two classes of functions, where $\linspan$ denotes the linear span:
      \[ \mcG := \linspan\{G_u\mid u\in\R^D\},\qquad
	      \mcF := \linspan\{F_u \cdot F_v\mid u,v\in\R^D\}. \]
      Then with \eqref{e.FT} and the uniqueness of $b$ and $C$ from Step 1, we have
      $ \int G \,\dx\mu_1 = \int G \,\dx\mu_2$ for all $G\in\mcG$.
      Now observe that, using linearity of $u\mapsto \psi_u(x)$ for every $x\in \R^D$,
      \begin{equation}\label{e.tr72}
	      F_u\cdot F_v = F_{u+v}-F_u-F_v = G_{u+v} - G_u - G_v \in \mcG \qquad\forall u,v\in \R^D.
      \end{equation}
      Hence, $\mcF$ is multiplicatively closed (by the first equality) and $\mcF\subseteq\mcG$. In particular,
      \eqref{i.vmult} holds,
	      \[ \int f \,\dx\mu_1 = \int f \,\dx\mu_2 \qquad\forall f\in\mcF, \]
      and $\mu_1,\mu_2 \in \mcM_\mcF^\#$, because functions from $\mcG$ are integrable.
      Furthermore, $\mcF$ is contained in $\mcC_b(X)$ and \eqref{i.sp} and \eqref{i.vanish} are easily verified.
      % $\{F_u\mid u\in\R^d\}$ induces the Euclidean topology (\colorbox{yellow}{does it?}, would be clear
      % for $x\mapsto 1-e^{-u^tx}$ and $\R_+^d$), hence the same is true for $\mcF$ and \eqref{i.top} holds. 
      %(need \colorbox{yellow}{$\C$-valued formulation} for \eqref{i.bound}; maybe: with $f \in \mcF$ we also have $\bar{f} \in \mcF$. Then the condition would be $f\bar{f} = \abs{f}^2 \geq \delta^2$. For all other function $f_n$ we may split in real and imaginary part.)
      Thus $\mu_1=\mu_2$ follows from Corollary~\ref{c.main}.

      \smallskip
      
      Now we show the convergence result.
 \step First, $\Psi_n(u) = \log \E[ \exp (iu^t Z_n) ] \to \Psi(u)$ pointwise since $x \mapsto \exp(iu^t x)$ is a bounded continuous function.
 \step Recall that for the L\'evy-triple $(b,C,\mu)$ the linear part $b$ depends on the choice of the compensation function $\psi_u$, but $C$ and $\mu$ do not.
      So in order to show $\mu_n \towh \mu$ we may choose any (admissible) $\psi_u$ we like.
      Replace $\psi_u$ in \eqref{e.psi_u} by $\psih_u(x) = iu^tx h(x)$ for a $\mcC^1$-function $h\colon \R^D \to
      \R$ with $h (0) = 1$ and compact support.
      Then the argument from above still works: $\mcF$ is multiplicatively closed and $\mcF \subset \mcC_b(X)$, so \eqref{i.mult} holds.
	
      Moreover, Assumption \eqref{i.top} holds by the fact that $F_u(x_n) \to F_u(x) \Leftrightarrow \exp(iu^t x_n) \to \exp(iu^t x)$ for all $u \in \R^D$.
      The latter is nothing else than the
      convergence of the characteristic function of the measures $\delta(x_n)$ to $\delta(x)$.
      But this implies that $x_n \to x$ in $\R^D$, so \eqref{i.top} holds.
      
      Finally, let $A \subset \R^D \setminus \{0\}$ be bounded w.r.t.~$d$.
      Then there is $\eps >0$ s.t.~$\eps < \inf \{\norm{x}_\infty \mid x \in A\} \leq \sup \{\norm{x}_\infty \mid x \in A\} < \eps^{-1}$.
      Consider $u = (u^\ast, \dots, u^\ast) \in \R^D$ with $u^\ast = (\eps \pi)/(2D)$.
      Then $u^t x \in ( \pi \eps^2/(2D), \pi/2)$ for $x \in A$ and moreover:
      \begin{align*}
	\inf_{x \in A} \abs{ F_u(x) }^2 & = \inf_{x\in A}\abs{ e^{iu^t x} - 1}^2 \geq \inf_{x \in A} \abs{ \cos(u^t x) - 1 }^2  \\
	& = \inf_{z \in  ( \pi \eps^2/(2D), \pi/2) } \abs{ \cos(z) -1 }^2 = \left( 1- \cos (\pi \eps^2 /(2D)) \right)^2 =: \delta^2 . 
      \end{align*}
      Thus, \eqref{i.bound} holds and Theorem~\ref{t.main} applies to show $\mu_n \towh \mu$.
\end{proofsteps}
\end{proof}

\begin{remark}
 Of course, the previous result is well-known, see \cite[Theorem 8.7]{Sato}.
\end{remark}

%**************************************** Examples: measure-valued Levy-Khintchine ****************************************
\subsection{Example 2: L\'evy-Khintchine formula on \texorpdfstring{$\mcM_f(E)$}{Mf(E)}}

Let $E$ be a Polish space and $\mcM_f(E)$ denote the finite measures on $E$.
Suppose that $Z$ takes values in finite measures on $E$ ($E= \{1,\dotsc, D\}$ is a special case of Example~1 if we restrict to nonnegative random variables there) and that $Z$ is infinitely divisible.
Then, under the assumption that $\E[Z(E)] < \infty$, Theorem 6.1 of \cite{Kall83} states that
there exists $b \in \mcM_f (E)$ and $\mu \in \mcM^\# (\mcM_f(E) \setminus \{0\})$ such that
\begin{equation}\label{e.LK.meas}
 L(\phi) := - \log \E[ \exp(- \la \phi , Z \ra) ] = \la \phi , b \ra + \int  1- \exp(-\la \phi , \nu \ra)
 \> \mu(\dx \nu), \quad \phi \in \mcC_{b}\bigl(E, [0,\infty)\bigr),
\end{equation}
where the set $\mcC_{b}\bigl(E,[0,\infty)\bigr)$ denotes the nonnegative, bounded, continuous functions on $E$.
We use the metric space $(X,d) = (\mcM_f(E) \setminus \{0\}, d)$ with 
 \[ d(\nu,\nu') = d_{\textrm{Prohorov}}(\nu,\nu') + | \nu(E)^{-1} - \nu'(E)^{-1}| \]
%This formula also goes under the name L\'evy-Khintchine formula.
for the definition of $\mcM^\#(X)$.
The uniqueness of the pair $(b,\mu)$ in the L\'evy-Khintchine formula in \eqref{e.LK.meas} can be shown with our methodology.
\begin{proposition}\label{p.LK.meas}
 The pair\/ $(b,\mu) \in \mcM_f (E) \times \mcM^\# (\mcM_f(E) \setminus \{0\})$ in \eqref{e.LK.meas} is unique.
\end{proposition}

\begin{proof}
 \begin{proofsteps}
 \step First, $b$ can be identified via $ \la b, \phi \ra = \lim_{m\to \infty} \frac{1}{m} L(m \phi)$. 
%    \[ \la b, \phi \ra = \lim_{m\to \infty} \frac{1}{m} L(m \phi) . \]
 \step To identify $\mu$ we want to use Corollary~\ref{c.main}.
    Since $(X,d)$ is a Polish space it is also a Souslin space.
    Define for $\phi \in \mcC_{b}(E, [0,\infty))$ the function $F_\phi: X \to [0,1]$ via
    \begin{equation*}
    F_\phi (\nu) = 1- \exp(- \la \phi, \nu \ra), \ \nu \in X .
    \end{equation*}
    The linear span of functions of this kind is defined
    \begin{equation*}
    \mcF :=   \linspan \{ F_\phi \mid \phi \in \mcC_{b}(E, [0,\infty)) \} .
    \end{equation*}
    Then it is easy to see that $\mcF \subset \mcC_b(X).$
 \step
    Now we want to verify the remaining conditions of Corollary~\ref{c.main}.
    \eqref{i.vmult} holds since $\mcF$ are real-valued functions and $F_\phi \cdot F_\psi = F_{\phi + \psi} - F_\phi -F_\psi$ for $\phi, \psi \in \mcC_b(E,[0,\infty))$.
    \eqref{i.sp} and \eqref{i.vanish} trivially hold.
    So we can apply the corollary and deduce the uniqueness of $\mu$.
 \end{proofsteps}
\end{proof}

\begin{remark}
 The previous proof also works with Theorem~\ref{t.main} and thus
 allows to deduce a result on the convergence of the characteristics for sequences of infinitely divisible random measures.
\end{remark}

%**************************************** Example: Excursion measure ****************************************
\subsection{Example 3: Excursion measure of Brownian motion}
Let $P_x \in \mcM_1\bigl(\mcC(\R_+,\R)\bigr)$ be the law of a $1$-dimensional Brownian motion started in $x \in \R$ and denote the canonical process by $(B_t)_{t\geq 0}$.
Let $T_0 := \inf \{t> 0:\, B_t = 0\}$ be the first hitting time of the origin $0$.
It is a folklore fact that the measure $\mu_n := n P_{1/n} ( (B_{t\wedge T_0})_{t\geq 0} \in \cdot )$ converges,
as $n\to \infty$, to the It\^o excursion measure $\muex$ of the reflected Brownian motion $(|B_t|)_{t\ge0}$. 
There are several ways to define $\muex$. We use the characterisation given in Theorem~XII.4.2 of
\cite{RY99} (where $\muex=2 n_+$ for $n_+$ used in \cite{RY99}) as definition.

\begin{definition}[Brownian excursion measure $\muex$]
	Let $X':=\mcC(\R_+; \R_+)$ be equipped with the topology of uniform convergence on compacta.
	For $r>0$, let\/ $\nu_r\in\mcM_1(X')$ be the law of a 3-dimensional Bessel bridge of length $r$.
	Define
	\[ \muex := \int_{\R_+} \nu_r\, \kappa(r)\,\dx r \qquad\text{for}\qquad
		\kappa(r) := (2\pi r^3)^{-1/2}. \]
	Then the $\sigma$-finite measure $\muex$ on $X'$ is called Brownian excursion measure. 
\end{definition}

Since $\muex$ is obviously not a finite measure, we have to be more precise about what we mean by convergence of
$\mu_n$ to $\muex$. 
In Theorem~1 of \cite{Hutz09} it is shown (for a more general class of diffusions) that
$\int F \,\dx \mu_n \to \int F \,\dx \muex$ holds for every $F\in\mcC_b(X')$ with the property that there is an
$\eps>0$ with $F(e)=0$ whenever $\|e\|_\infty < \eps$. This looks very much like \whconv\ on $X'\setminus\{0\}$,
where $0$ denotes the zero function and is sent infinitely far away.\footnote{Given a metric space $(X,d)$ and
$x\in X$, ``sending $x$ infinitely far away'' is a figure of speech for considering $X\setminus\{x\}$ with a
metric $d'$ topologically equivalent to $d$, but making every sequence that $d$-converges to $x$ leave every
$d'$-ball. A possible choice is $d'(y,z)=d(y,z)+|d(x,y)^{-1}-d(x,z)^{-1}|$. Formally, $d'(x,y)=\infty$ for all
$y\in X\setminus\{x\}$.}
This is, however, not precisely the case, because the map $e \mapsto \|e\|_\infty$ is not continuous w.r.t.\
uniform convergence on compacta.

In this subsection we give a setup, where we can apply Theorem~\ref{t.main} to obtain a \whconv\ $\mu_n\towh
\muex$. To this end, we have to modify the topology on (a subspace of) $X'$ in two ways. First, we
weaken uniform convergence on compacta to convergence in Lebesgue measure, because the latter is induced by
``nice'' functions and therefore much easier to handle in our framework. This, of course, substantially weakens
our result, so that it does not imply the one in \cite{Hutz09}. Second, we
strengthen the topology (and therefore our result) a bit by additionally requiring convergence of excursion
lengths for the convergence of excursions.
This allows us to send the zero function infinitely far away by a continuous function, and the result in
\cite{Hutz09} does not directly include ours.

\begin{definition}[Our excursion space]\label{d.exspace}
	Define the excursion length $\zeta\colon X' \to [0,\infty]$ by
	\[ \zeta(e) = \sup \{t>0\mid e(t) \ne 0\} \cup \{0\}, \]
	the space of excursions $X:=\zeta^{-1}\bigl((0,\infty)\bigr)$, and the metric
	\[ d(e, \eh) = \Bigl( \int_0^\infty \abs{e(t) - \eh(t)}\land 1 \,\dx t\Bigr) \land 1
				+ \abs{\zeta(e)^{-1} - \zeta(\eh)^{-1}} \]
	on $X$.
\end{definition}

The topology induced by $d$ on $X$ is the Meyer-Zheng topology (or pseudo-path topology) introduced in
\cite{MeyerZheng84} plus convergence of excursion lengths as we show in Lemma~\ref{l.dtop}.

\begin{definition}[Meyer-Zheng topology]
	Let\/ $\lambda$ be the probability measure on $\R_+$ with Lebesgue-density $t\mapsto e^{-t}$,
	and $e_n,e\colon \R_+ \to \R_+$ measurable. Then $e_n$ is said to converge to $e$ in Meyer-Zheng
	topology if the image measures of $\lambda$ under $e_n$ converge weakly to the one under $e$.
\end{definition}

\begin{lemma}[The topology induced by $d$]\label{l.dtop}
	Let\/ $e_n,e\in X$. Then the following are equivalent:
	\begin{enumerate}
		\item\label{i.dconv} $e_n \to e$ with respect to\/ $d$.
		\item\label{i.Lconv} $e_n$ converges to\/ $e$ in Lebesgue-measure, and\/ $\zeta(e_n) \to \zeta(e)$.
		\item\label{i.MZconv} $e_n$ converges to\/ $e$ in Meyer-Zheng topology, and\/ $\zeta(e_n) \to \zeta(e)$.
	\end{enumerate}
	In particular, $(X,d)$ is a separable metric space.
\end{lemma}
\begin{proof}
	For the ``in particular'' note that separability of convergence in Lebesgue-measure is well-known and
	carries over to $d$-convergence by the first equivalence.

	\noindent \ref{i.dconv}$\Leftrightarrow$\ref{i.Lconv}:
	Since $\zeta(e) \in (0,\infty)$ for all $e\in X$, we have that $\abs{\zeta(e)^{-1} - \zeta(\eh)^{-1}}
	\to 0$ is equivalent to $\zeta(e_n) \to \zeta(e)$.
	It is well-known that $d_{\mathrm{Leb}}(e,\eh)=\int \abs{e(t) - \eh(t)}\land 1 \,\dx t$ induces convergence in
	Lebesgue-measure (e.g.\ \cite[Exercise~4.7.61]{BogachevI}), so the same is true for $d_{\mathrm{Leb}}\land 1$.

	\noindent \ref{i.MZconv}$\Leftrightarrow$\ref{i.Lconv}:
	In \cite[Lemma~1]{MeyerZheng84} it is shown that Meyer-Zheng topology coincides with convergence in
	$\lambda$-measure.
	Now $\zeta(e_n)\to \zeta(e)$ implies $M:=\sup_{n\in\N} \zeta(e_n) < \infty$, and $\lambda$ is equivalent
	to Lebesgue-measure on $[0,M]$. 
\end{proof}

\begin{theorem}[Brownian excursion measure]\label{t.ex}
 Let\/ $(X,d)$ be the excursion space introduced in Definition~\ref{d.exspace}, $\muex$ the Brownian excursion
 measure, and\/ $\Bh=(\Bh_t)_{t\ge0}$ Brownian motion killed in\/ $0$. In\/ $\mcM^\#(X)$,
 \[  \mu_n := n P_{1/n}(\Bh \in \cdot) \towh \muex  \quad\text{ as\/ } n \to \infty. \]
\end{theorem}

In order to use Theorem~\ref{t.main} to prove Theorem~\ref{t.ex}, we need a set $\mcF$ of continuous functions on $X$
satisfying \eqref{i.mult} -- \eqref{i.bound}.
To this end, define for $f \in L^1(\R_+)$ and $g \in \mcC_b(\R_+)$
\[ F_{f,g} (e) = \int_0^\infty f(t) g(e(t)) \, \dx t.  \]
Denote by $\mcC_c$ the continuous functions with compact support and define a set of bounded functions on $X$ by
\[ \mcF' =  \left\{  F_{f,g}  \mid f \in \mcC_c(\R_+) ,\, x \mapsto (1\wedge x)^{-1}g(x) \in
\mcC_b(\R_+) \right\} \cup \{h \circ \zeta \mid h \in \mcC_b(\R_+)\}. \]

\begin{definition}[$\mcF$]
%	Let $\mcF = \alg(\mcF')$ be the algebra generated by $\mcF'$.
	Let $\mcF$ be the multiplicative closure of $\mcF'$.
\end{definition}

\begin{lemma}\label{l.convdet}
	$\mcF$ is weak$^\#$-convergence determining for measures in\/ $\mcM_\mcF^\#(X)$.
\end{lemma}
\begin{proof}
	$\mcF$ obviously satisfies \eqref{i.mult} and \eqref{i.bound}. Indeed,
	$\mcF$ is multiplicatively closed by definition, and $A\subseteq X$ is $d$-bounded if and only if
	$\inf_{e\in A} \zeta(e) > 0$. 
	Once we have shown \eqref{i.top}, i.e.\ that $\mcF$ induces the same topology as $d$, the claim follows
	from Theorem~\ref{t.main}. To this end, note that $\zeta(e_n) \to \zeta(e)$ is necessary for both $d$-
	and $\mcF$-convergence, and recall that under this condition, by Lemma~\ref{l.dtop}, $d$-convergence is
	equivalent to
	\begin{equation}\label{e.MZconv}
	 \int_0^\infty \varphi(t, e_n(t))\, \lambda(\dx t) \to \int_0^\infty \varphi(t, e(t))\, \lambda(\dx t)
		\qquad\forall \varphi \in \mcC_b(\R_+^2).
	\end{equation}
	The set of $\varphi_{f,g}$ of the form $\varphi_{f,g}(t,x)=f(t)g(x)e^t$ with $f\in\mcC_c(\R_+)$, $g\in
	\mcC_b(\R_+)$ and $x\mapsto x^{-1}g(x) \in \mcC_b(\R_+)$ is a multiplicatively closed subset of
	$\mcC_b(\R_+^2)$, and induces the Euclidean topology on $\R_+^2$. Thus, by the classical Le Cam
	theorem, it is convergence determining, and \eqref{e.MZconv} is equivalent to
	\[ F_{f,g}(e_n)=\int_0^\infty \varphi_{f,g}(t, e_n(t))\, \lambda(\dx t)  \to F_{f,g}(e) \qquad \forall F_{f,g} \in \mcF', \]
	which implies the claim.
%	 it is enough to show for $e_n,e\in X$
%	\[ F(e_n) \to F(e) \ \forall F \in \mcF' \;\Leftrightarrow\; d(e_n,e)\to 0. \]
%	\zeta(e_n)\to\zeta(e),\; e_n\to e\text{ in Lebesgue-measure}. \]
%	\noindent $\Leftarrow$: Suppose $d(e_n,e) \to 0$.
%	Then $\zeta(e_n)  \to \zeta(e)$ and hence $h\circ\zeta(e_n) \to h\circ\zeta(e)$ for $h\in\mcC_b(\R_+)$.
%	Now fix $f\in \mcC_c$, $g\in \mcC$. For $\eps>0$, let
%	$\delta_\eps=\sup\{g(x)-g(y)\mid x,y\in [0,\|e\|_\infty+\eps],\; |x-y|<\eps\}$, and observe that
%	$\delta_\eps \to 0$ as $\eps \to 0$. Denoting Lebesgue-measure with $\lambda$, we obtain
%	\[ \abs{F_{f,g}(e_n) - F_{f,g}(e)} \le 2\lambda\bigl(\bigl\{ t \bigm| \abs{e_n(t)-e(t)} \ge \eps \bigr\}\bigr)
%		\cdot \|f\|_\infty \cdot \|g\|_\infty + \delta_\eps \|f\|_1, \]
%	where $\|f\|_1$ is the $L^1$-norm of $f$.
%	Letting first $n$ tend to infinity and then $\eps$ to zero, we see using Lemma~\ref{l.dtop} that
%	the right-hand side converges to zero.
\end{proof}

\begin{proof}[Proof of Theorem~\ref{t.ex}]
  In view of Lemma~\ref{l.convdet} it is sufficient to show
  \begin{equation}\label{e.toshow}
  	 \lim_{\eps\to0} \tfrac1\eps \Ex[\eps]{F(\Bh)} = \int F \,\dx\muex \qquad \forall F\in\mcF.
  \end{equation}
  Fix $F\in\mcF$. There is $h\in\mcC_b(\R_+)$, $n\in\N_0$, $f_i \in L^1 \cap \mcC_c(\R_+)$,
  $x \mapsto (1\wedge x) ^{-1}g_i(x) \in \mcC_b(\R_+)$ for $1\leq i \leq n$, such that
  \begin{equation}\label{e.Frep}
  	F(e) = h\circ \zeta(e) \cdot \prod_{i=1}^n F_{f_i,g_i}(e) = \int_{\R_+^n} f(\td) h(\zeta(e)) \get \,\dx\td,
  \end{equation}
  where we set $\td=(t_1,\ldots,t_n)$, $f(\td)=\prod_{i=1}^n f_i(t_i)$ and $g(\td) = \prod_{i=1}^n
  g_i(t_i)$. Let $\tmax=\max \{t_1,\ldots,t_n\}$. Using that $\zeta(e)=r$ $\nu_r$-a.s., and $g_i(e(t_i))=0$ whenever
  $t_i> \zeta(e)$, we obtain
  \begin{align}\label{e.intFmuex}
  \int F\,\dx\muex &= \int_0^\infty \kappa(r) h(r) \int \prod_{i=1}^n F_{f_i,g_i}\, \dx\nu_r \, \dx r \\
    &= \int_{\R_+^n} f(\td) \int_0^\infty h(\tmax+r)\kappa(\tmax+r) \int \get \,\nu_{\tmax+r}(\dx e) \,\dx r\,\dx\td
    \nn\\
    &= \int_{\R_+^n} f(\td) \int_0^\infty h(\tmax+r) \bEx[0]{
    	\frac{g(\rho_{t_1},\ldots,\rho_{t_n})}{\rho_{\tmax}}{\ell^{\rho_{\tmax}}(r)}} \,\dx r\,\dx\td,\nn
  \end{align}
  where $\rho=(\rho_t)_{t\ge 0}$ denotes a 3-dimensional Bessel process, $\ell^\alpha$ the density of a L\'evy
  distribution with scale parameter $\alpha$, and we have used the relation of densities of the Bessel bridge
  and process (as obtained, e.g., in \cite[XI.\S3]{RY99}).

  On the other hand, recall that $T_0=\zeta(\Bh)$ is the hitting time of $0$ of $\Bh$ and observe
  \begin{equation}\label{e.intFB}
  	\tfrac1\eps\Ex[\eps]{F(\Bh)} = \int f(\td) \Ex[\eps]{h(T_0)\tfrac1\eps \prod_{i=1}^n g_i(\Bh_{t_i})} \,\dx\td.
  \end{equation}
  For fixed $\td$, reordering if necessary, we assume for notational convenience that $t_1\le \cdots\le
  t_n=\tmax$ and set $t_0=0$. Then, using the
  Markov property at $t_n$ in the first step,
  \begin{align}
      \bEx[\eps]{h(T_0)\tfrac1\eps \prod_{i=1}^n g_i(\Bh_{t_i})} &=
      	  \bEx[\eps]{ \Ex[\Bh_{t_n}]{h(t_n+T_0)} \frac1{\Bh_{t_n}}\prod_{i=1}^n \frac{\Bh_{t_i}}{\Bh_{t_{i-1}}} g_i(\Bh_{t_i}) }
	  	\nonumber\\
       &= \bEx[\eps]{ \Ex[\rho_{t_n}]{h(t_n+T_0)} \frac1{\rho_{t_n}}\prod_{i=1}^n g_i(\rho_{t_i}) },
	       \label{e.fineps}
  \end{align}
  %$\rho=(\rho_t)_{t\ge 0}$ denotes a 3-dimensional Bessel process, and
  where we have used that the sub-Markovian semigroup $(Q_t)_{t\ge 0}$ of the killed Brownian motion $\Bh$ and
  the Markovian semigroup $(H_t)_{t\geq 0}$ of the Bessel process $\rho$ are related by
  \begin{equation}
   H_t(x,\dx y) = \begin{cases}
                    x^{-1}Q_t(x,\dx y) y & \text{ for } x >0,\\
                    2\kappa(t) y^2 \exp(-y^2/(2t)) \, \dx y & \text{ for } x= 0.
                   \end{cases}
  \end{equation}
  Because $\Bh$ is a Feller-process and $h\in\mcC_b$, the function $x\mapsto \Ex[x]{h(t_n+T_0)}$ is also a
  bounded continuous function. Because also $x\mapsto x^{-1}g_n(x)$ is bounded and continuous by assumption,
  we see that the term inside the outer expectation in \eqref{e.fineps} is a bounded continuous function in
  $\rho_{t_1}, \ldots, \rho_{t_n}$. Using that $\rho$ is a Feller process, we obtain from \eqref{e.intFB} and
  \eqref{e.fineps}
  \begin{equation}\label{e.conv}
    \lim_{\eps\to 0} \tfrac1\eps\Ex[\eps]{F(\Bh)} =
	\int f(\td) \bEx[0]{ \Ex[\rho_{\tmax}]{h(\tmax+T_0)} \frac{g(\rho_{t_1},\ldots,\rho_{t_n})}{\rho_{\tmax}} } \,\dx\td.
  \end{equation}
  The hitting time $T_0$ of $0$ under $P_x$ is known to be L\'evy-distributed with scale parameter $x$, hence
  \begin{equation}\label{e.hitdens}
  	\Ex[\rho_{\tmax}]{h(\tmax+T_0)} = \int_0^\infty h(\tmax+r)\ell^{\rho_{\tmax}}(r) \,\dx r.
  \end{equation}
  Inserting \eqref{e.hitdens} into \eqref{e.conv}, and applying \eqref{e.intFmuex}, we obtain the claimed
  convergence \eqref{e.toshow}.
\end{proof}

%**************************************** Example: Fragmentation ****************************************
\subsection{Example 4: Mass fragmentations}\label{ex.fragm}

%Fragmentation processes have played a prominent role in the last decade; Bertoin.
Consider the set
\[ S^\downarrow = \Bigl\{ \su = (s_1, s_2, \dots) \in [0,1]^\N \Bigm| \sum_{i \geq 1} s_i \leq 1, \; s_1 \geq s_2 \geq s_3 \geq \cdots \geq 0 \Bigr\} \]
of decreasing sequences $(s_1,s_2,\dotsc)$ with sum less than $1$.
This set is used to model (mass) fragmentation processes.
A perfect introduction to the topic is given in Bertoin's book \cite{Bertoin}.
In his Section 2.1, the set $S^\downarrow$ is introduced and endowed with the topology of pointwise convergence.
Our goal is to present two sets of real-valued convergence-determining functions on
$S^\downarrow$, namely
% Sometimes the the additional entry $s_0 = 1 - \sum_{i\geq 1} s_i$ is added, but we refrain from doing so here.
\begin{align}
  \mcG_1 &= \Bigl\{ S^\downarrow \ni \su = (s_1,s_2, \dotsc) \mapsto G_p(\su) = \sum_{i\geq 1} s_i^p 
 \Bigm| p \in \N \Bigr\}, \\
 \mcG_2 &= \Bigl\{ S^\downarrow \ni \su = (s_1,s_2, \dotsc) \mapsto H_\alpha(\su) = \sum_{i\geq 1}
  (1-e^{-\alpha s_i}) \Bigm| \alpha > 0 \Bigr\} .
\end{align}

\begin{theorem}\label{c.fragm.sep}
 The sets\/ $\mcG_1$ and\/ $\mcG_2$ are convergence-determining on\/ $S^\downarrow$ in the sense that for\/
 $\su(n), \su \in S^\downarrow$, $n \in \N$ the following holds for\/ $j = 1,2$:
 \[  G(\su(n)) \xrightarrow{n\to \infty} G(\su) \; \forall G \in \mcG_j \quad\implies\quad  \su(n) \xrightarrow{n\to \infty} \su . \]
\end{theorem}

\noindent
Before we give the proof, we relate the set $S^\downarrow$ to boundedly finite measures.
Therefore, let 
\[ X = (0,1]  \text{ with metric } d(x,y) = \abs{ x^{-1}-y^{-1} } . \]
Consider the mapping
\[ \Phi: \begin{cases}      S^\downarrow &\to \mcM^\#(X) , \\
                        (s_1, s_2 \dots) &\mapsto \sum_{i\geq 1} \delta_{s_i} .
         \end{cases}
\]

\begin{lemma}\label{p.fragm}
 The mapping\/ $\Phi$ is a homeomorphism from\/ $S^\downarrow$ to\/ $\Phi(S^\downarrow)$.
\end{lemma}
\begin{proof}
\begin{proofsteps}
 \step The mapping $\Phi$ is injective: For $\su \in S^\downarrow$, all $s_i, i \geq 1$ can be easily reconstructed.
%  \[ s_0 = 1- \sum_{i\geq 1} s_i = 1- \int \Phi(\su)(\dx x) \, x . \]
 \step The mapping $\Phi$ is continuous: Let $\su(n), \su \in S^\downarrow$ with $\su(n) \to \su$.
 If $A$ is a bounded set in $X$ w.r.t.~$d$ with $\Phi(\su)(\partial A) = 0$, then it is clear that $ \Phi(\su(n))(A) \to \Phi(\su)(A)$ and we can use Proposition A.2.6.II~(d) in \cite{DVJ03}.
 
 \step The mapping $\Phi^{-1}|_{\Phi(S^\downarrow)}$ is continuous: Suppose $\Phi(\su(n)) \towh \Phi(\su)$.
 Fix $z \in (0,1)$ with $z \not \in \{\su_i \mid i \in \N\}$.
 Then we have $\Phi(\su(n))|_{[z,1]} \xrightarrow{w} \Phi(\su)|_{[z,1]}$.
 But this implies convergence of $\su(n)$ to $\su$ on those coordinates that lie in $[z,1]$.
 Since we may choose $z$ arbitrarily small, this suffices.
\end{proofsteps}
\end{proof}

%Now we can provide the proof of the theorem.
\begin{proof}[Proof of Theorem~\ref{c.fragm.sep}]
We only provide the proof for $\mcG_1$, since the other proof is similar.
 \begin{proofsteps}
 \step Note that we can write
      \[ G_p(\su) =  \int x^p \, \Phi(\su)(\dx x), \ p \in \N. \]
    Thus, $G(\su(n)) \to G(\su) \, \forall G \in \mcG_1$ can be written as
      \[ \int x^p \, \Phi(\su(n))(\dx x) \to \int x^p \, \Phi(\su)(\dx x), \ p \in \N .\]
    We use Theorem~\ref{t.main} to show that $\Phi(\su(n)) \towh \Phi(\su)$ and this suffices for $\su(n) \to \su$ by Lemma~\ref{p.fragm}.
 \step Of course $(X,d)$ as before Lemma~\ref{p.fragm} is a separable metric space.
    Moreover, the class $\mcF = \linspan \{x \mapsto x^p \mid p \in \N\} \subset \mcC_b(X)$ of polynomials
    without constant term on $X$ satisfies the prerequisites of Theorem~\ref{t.main}, which is easy to check.
    Finally,
	$\Phi(\su) \in \mcM_{\mcF}^\#(X)$ for all $\su \in S^\downarrow$ since for all $p \in \N$:
	\[ \int x^p \, \Phi(\su) (\dx x) \leq \int x \, \Phi(\su)(\dx x) = \sum_{i\geq 1} s_i \leq 1 < \infty .\]
    Thus, Theorem~\ref{t.main} applies and yields the claim.
 \end{proofsteps}
\end{proof}

\begin{remark}
 Note that $\mcG_1 \not \subset \mcC(S^\downarrow)$: for $\su_i(n) = n^{-1} \1_{1\leq i \leq n}$, $i \in \N$, we have $\su(n) \to \underline{0}$, but $G_1(\su(n)) = 1 \not \to 0 = G_1(\underline{0})$.
\end{remark}

The stronger statement than Theorem~\ref{c.fragm.sep} including the continuity holds on the subset of decreasing sequences summing up to 1.

\begin{corollary}\label{c.fragm.top}
 On the subset\/ $S_1^\downarrow = \{\su \in S^\downarrow \mid \sum_{i\geq 1} s_i = 1 \}$, the set of
 functions\/ $\mcG_1$ generates the topology of pointwise convergence.
\end{corollary}
\begin{proof}

 First, the subset is a Polish space with the relative topology since it is a  $G_\delta$-subset of $S^\downarrow$.
 
 Let $\su(n), \su \in S_1^\downarrow$, $n \in \N$.
\begin{proofsteps}
 \step First suppose that $\su(n) \to \su$ as $n \to \infty$.
 We can use an approximation argument to get $G(\su(n)) \to  G(\su)$ for all $G \in \mcG_1$:
 Let $\eps >0$.
 Fix $z > 0$ s.t.~$\sum_{i\geq 1} s_i \1_{s_i \leq z} < \eps$ and $s_i \neq z$ for all $i \geq 1$.
 By Proposition A.2.6.II~(d) from \cite{DVJ03} and our Proposition~\ref{p.fragm}, we may choose $n$ so large that $\Phi(\su(n))|_{[z,1]}$ and $\Phi(\su)|_{[z,1]}$ are very close in the following sense:
 \[ \sup_{q \in \{1,p\}} \Bigl| \sum_{i \geq 1} s_i(n)^q\, \1_{s_i(n) > z} - \sum_{i \geq 1} s_i^q\, \1_{s_i(n) > z} \Bigr| = A(\eps) < \eps . \]
 Use that in the following equation
 \begin{align}
  G_p(\su(n)) & = \sum_{i \geq 1} s_i(n)^p\, \1_{s_i(n) > z} + \sum_{i\geq 1} s_i(n)^p\, \1_{s_i(n) \leq z} \\
  & = A(\eps) +  \sum_{i \geq 1} s_i^p\, \1_{s_i > z} + \sum_{i\geq 1} s_i(n)^p\, \1_{s_i(n) \leq z} .
 \end{align}
 Once we establish that the last term is small, we see that $G_p(\su(n)) \to G_p(\su)$.
 But the last term is bounded by:
 \begin{equation}
 	\sum_{i\geq 1} s_i(n)^p\, \1_{s_i(n) \leq z}  \leq \sum_{i\geq 1} s_i(n) \1_{s_i(n) \leq z} = 1 - \sum_{i\geq 1} s_i(n) \1_{s_i(n) > z} \leq 1- (1-\eps) = \eps.
 \end{equation}
% \begin{align}
%  \sum_{i\geq 1} s_i(n)^p\, \1_{s_i(n) \leq z} & \leq \sum_{i\geq 1} s_i(n) \1_{s_i(n) \leq z} = 1 - \sum_{i\geq 1} s_i(n) \1_{s_i(n) > z} \\
%  & \leq 1- (1-\eps) = \eps.
% \end{align}
 Note that it was crucial to know that $\sum_{i \geq 1} s_i =1$ in the last proof.

 \step The converse direction was established in Theorem~\ref{c.fragm.sep}.
\end{proofsteps}
\end{proof}

\begin{remark}
 An application in a similar spirit is given in \cite{infdiv}, where
 $X$ is the set of ultrametric measure spaces with diameter in $[0,2h)$ for certain $h>0$.
 Any ultrametric measure space with diameter in $[0,2h]$ can be written as a boundedly finite measure on $X$ (in a unique way similar to a prime factorization).
 This relation can be used for a result analogous to Corollary~\ref{c.fragm.top}.
\end{remark}

\bibliography{ggr}
\bibliographystyle{amsalpha}

\end{document}